\documentclass[a4paper]{amsart}
\usepackage{amsmath}
\usepackage{graphicx}
\usepackage{amssymb}
\usepackage{mathrsfs}
\usepackage{nicematrix}
\usepackage{stmaryrd}
\usepackage[mathcal]{euscript}
\usepackage{tikz}
\usepackage{tikz-cd}
\usepackage{hyperref}
\hypersetup{%
  bookmarksnumbered=true,%
  colorlinks=true,%
  linkcolor=blue,%
  citecolor=blue,%
  filecolor=blue,%
  menucolor=blue,%
  urlcolor=blue,%
  bookmarksopen=true,%
  bookmarksdepth=2,%
  pageanchor=true}

\title[Multisets of finite intervals]{Multisets of finite intervals and a
  universal category of poset representations}

\author{Henning Krause}
\author{Balduin Stoye}

\address{Fakult\"at f\"ur Mathematik\\
Universit\"at Bielefeld\\ D-33501 Bielefeld\\ Germany}

\theoremstyle{plain}
\newtheorem{thm}{Theorem}
\newtheorem{prop}[thm]{Proposition}
\newtheorem{lem}[thm]{Lemma} 

\newtheorem{cor}[thm]{Corollary}

\theoremstyle{definition}
\newtheorem{defn}[thm]{Definition}
\newtheorem{exm}[thm]{Example}

\theoremstyle{remark}
\newtheorem{rem}[thm]{Remark}

\numberwithin{equation}{section}

\newcommand{\card}{\operatorname{card}}

\newcommand{\Coker}{\operatorname{Coker}}

\newcommand{\Ext}{\operatorname{Ext}}

\newcommand{\Hom}{\operatorname{Hom}}

\newcommand{\Ker}{\operatorname{Ker}}

\renewcommand{\min}{\operatorname{min}}
\renewcommand{\mod}{\operatorname{mod}}
\newcommand{\Mod}{\operatorname{Mod}}

\newcommand{\NC}{\operatorname{NC}}
\newcommand{\Ob}{\operatorname{Ob}}

\newcommand{\rad}{\operatorname{rad}}

\newcommand{\rep}{\operatorname{rep}}
\newcommand{\Rep}{\operatorname{Rep}}

\newcommand{\Ab}{\mathrm{Ab}}

\newcommand{\op}{\mathrm{op}}

\newcommand{\Set}{\mathrm{Set}}

\newcommand{\iso}{\xrightarrow{\raisebox{-.4ex}[0ex][0ex]{$\scriptstyle{\sim}$}}}

\newcommand{\longiso}{\xrightarrow{\ \raisebox{-.4ex}[0ex][0ex]{$\scriptstyle{\sim}$}\ }}

\newcommand{\lto}{\longrightarrow}

\newcommand{\xto}{\xrightarrow}
\newcommand*{\intref}[2]{\def\tmp{#1}\ifx\tmp\empty\hyperref[#2]{\ref*{#2}}\else\hyperref[#2]{#1~\ref*{#2}}\fi}

\def\A{\mathcal A} 
 
\def\C{\mathcal C}

\def\I{\mathcal I}

\def\bbF{\mathbb F} 
\def\bbK{\mathbb K}

\def\bbZ{\mathbb Z}

\begin{document}

\begin{abstract}
  For any finite totally ordered set, the multisets of intervals form
  an abelian category. Various classes of subcategories admit natural
  combinatorial descriptions, and counting them yields familiar
  integer sequences, including Catalan and large Schröder
  numbers. Surprisingly, in some cases new integer sequences
  arise. The formulation of this counting problem leads to a universal
  construction which assigns to any poset a finitely cocomplete
  additive category; it is abelian when the poset is finite and does
  not depend on the choice of any ring of coefficients. For a general
  poset the universal category of representations is abelian if and
  only if for the lattice of ideals the meet of two compact elements
  is again compact.
\end{abstract}



\date{\today}

\maketitle

\setcounter{tocdepth}{1}
\tableofcontents

\section{Introduction}

The study of quiver representations involves the analysis of
subcategories that are defined by natural closure operations; see for
instance \cite{BK2004,IT2009,ONFR2015, ORT2015,Th2012}. It is an
obvious task to count them when there are only finitely many.  In this
note we provide a systematic account for the quiver of type $A_n$ with
linear orientation ($n\ge 1$). This may be formulated as a purely
combinatorial problem. To this end we consider for any finite totally
ordered set the category of multisets of intervals
(Definition~\ref{de:intervals}). It has been known for a long time
that Catalan combinatorics plays an important role \cite{Ga1981}.
Beyond that we discuss some new phenomena. In particular, we clarify
the role of large Schröder numbers; their relevance was noticed before
by Enomoto \cite{En2021,En2022}, but also in \cite{CFR2013,Ro2021}.

The category of multisets of intervals is an abelian category which
can be obtained from a universal construction.  In fact, we construct
for any poset $P$ a universal functor $P\to \rep P$ into a finitely
cocomplete category with a zero object
(Proposition~\ref{pr:cocomplete}).  When $P$ is finite the category
$\rep P$ is actually a \emph{length category}, so an abelian category
such that each object has a finite composition series.  We note that
the construction of $\rep P$ does not depend on the choice of any ring
of coefficients. Thus $\rep P$ may be thought of as a \emph{universal
  category of poset representations}.  It seems to be an interesting
project to study properties of $P$ in terms of $\rep P$.  A striking
example is a recent result of Iyama and Marczinzik showing that a
finite lattice $P$ is distributive if and only if $\rep P$ is the
module category of an Auslander regular algebra \cite{IM2022}. We end
this note by characterising the posets such that $\rep P$ is an
abelian category (Proposition~\ref{pr:coherent}). Along the way we
characterise the posets such that all objects in $\rep P$ are either
noetherian or of finite length (Corollary~\ref{co:subobj}).

\section{Finite intervals}

Let $\bbZ$ denote the set of integers and let $\bbF_2$ denote the field
with two elements. We consider categories that are
\emph{$\bbF_2$-linear}, so the morphism sets are $\bbF_2$-linear
spaces and the composition maps are bilinear.  An \emph{interval of
  length $l$} is a pair of integers $a\le b$ such that $b-a=l$.

\begin{defn}
  For $n\ge 1$ let $\I_n$ denote the category of \emph{intervals
    of type $n$}. The objects are given by
\[\Ob\I_n:=\{[a,b]\in\bbZ\times\bbZ \mid 1\le a\le b\le n\}.\] For objects
$x=[a,b]$ and $y=[c,d]$ set
\[\Hom_{\I_n}(x,y):=\begin{cases} \bbF_2&\text{if }a\le c\le b\le d,\\
    \{0\}& \text{otherwise}.
\end{cases}\]
In case $\Hom_{\I_n}(x,y)\neq 0$  we write $\alpha_{xy}$ for the non-zero morphism
$x\to y$. The composition of morphisms in $\I_n$ is determined by two
rules:
\[\alpha_{yz}\circ\alpha_{xy}=\alpha_{xz}\]  when
\[x=[a,b],\, y=[c,d],\, z=[e,f]\qquad\text{and}\qquad a\le c\le e\le
  b\le d\le f.\] And $\psi\circ\phi=0$ for any other pair of morphisms
$\phi\colon x\to y$ and $\psi\colon y\to z$.
\end{defn}

Next we consider multisets of intervals, and we use the suggestive
notation
\[\oplus_{i=1}^r x_i:=\{x_1,\ldots,x_r\}\]
for a finite multiset of intervals $x_i=[a_i,b_i]$.

\begin{defn}\label{de:intervals}
  For $n\ge 1$ let $\A_n$ denote the category of \emph{mutisets of
  intervals of type $n$}. The objects are given by
\[\Ob\A_n:=\{\oplus_{i=1}^r x_i\mid x_i\in\I_n\text{ for }1\le
  i\le r\}\]
and the morphism sets are $\bbF_2$-linear spaces given by 
\[\Hom_{\A_n}(\oplus_i x_i, \oplus_j y_j):=\bigoplus_{i,j}\Hom_{\I_n}(x_i,y_j).\]
The composition of morphisms in $\A_n$ is given by
\[(\psi_{jk})\circ(\phi_{ij}):=\left(\sum_j\psi_{jk}\circ\phi_{ij}\right)\]
  for $(\phi_{ij})\colon \oplus_i x_i\to \oplus_j y_j$ and
  $(\psi_{jk})\colon \oplus_j y_j\to \oplus_k z_k$.
\end{defn}

\begin{lem}\label{le:repAn}
  The category $\A_n$ is abelian. The direct sum of objects is given
  by the disjoint union, and the full subcategory of indecomposable
  objects identifies with $\I_n$.  Every object admits a finite
  composition series, and the composition length of an interval of
  length $l$ equals $l+1$.
\end{lem}
\begin{proof}
We consider the quiver
\[A_n\colon\quad 1\lto 2\lto\cdots\lto n\] and write
$\rep(A_n,\bbF_2)$ for the category of $\bbF_2$-linear
\emph{contravariant representations} of $A_n$ that are finite
dimensional.  This is an $\bbF_2$-linear abelian category. In $\rep(A_n,\bbF_2)$ each object
decomposes essentially uniqueley into indecomposable representations
by the Krull--Remak--Schmidt theorem.
Up to isomorphism, the indecomposable representations are given by
intervals $x=[a,b]\in\I_n$ and are of the form
\[M_x\colon \quad 0\lto\cdots \lto 0\lto \bbF_2\xto{\ 1\ }\cdots\xto{\ 1\
  }\bbF_2\lto 0\lto\cdots\lto 0\] such that
\[M_x(i):=\begin{cases} \bbF_2&\text{if }a\le i\le b,\\
    \{0\}& \text{otherwise}.
  \end{cases}\]
The assignment $\oplus_i x_i\mapsto \bigoplus_i M_{x_i}$ yields an equivalence $\A_n\iso
\rep(A_n,\bbF_2)$, and from this the assertion of the lemma follows.
\end{proof}

\section{Counting subcategories}

For $n\ge 1$ we consider full subcategories $\A\subseteq\A_n$ that are
closed under finite direct sums and summands. Such  a subcategory is
determined by $\A\cap\I_n$. There are  several classes of
subcategories which are defined by natural closure
properties. We consider the following operations and combinations thereof:
\begin{enumerate}
\item[(Q)] taking quotients,
\item[(S)] taking subobjects,
\item[(C)] taking cokernels,
\item[(K)] taking kernels,
\item[(E)] taking extension.
\end{enumerate}
In each case we compute the number of such
subcategories, depending on $n$. This yields an integer
sequence, which we denote by \[\#(X,Y,Z)\] when
the relevant operations are  (X), (Y), and (Z).
We provide the first six terms and the formula for
computing the $n$-th term in case it is known. In some cases
the $n$-th \emph{Catalan number} and
the $n$-th \emph{large Schröder number} arise, which are defined as follows:
\[C_n := \frac{1}{n+1}\binom{2n}{n}\qquad\text{and}\qquad S_n :=
  \sum_{i=0}^n\frac{1}{i+1}\binom{n}{i}\binom{n+i}{i}\]

Not all combinations of the above operations need to be considered,
because some operations imply each other and some others are dual to
each other. For instance, any cokernel is a quotient, or a cokernel is
the dual of a kernel. Observe that there is a duality
\begin{equation}\label{eq:dual}
  (\A_n)^\op\longiso\A_n
\end{equation}  given by the involution of $\{1,2,\ldots,n\}$
taking $i$ to $n-i+1$.

\begin{rem}
  (1) Subcategories are partially ordered by inclusion, and counting
  them often amounts to an isomorphism of posets when one identifies
  them with some other combinatorial structure. The subcategories
  satisfying any choice of the above closure properties are closed under
  intersections; so they form a lattice.

  (2)  Let $\bbF$ be any field. Then one may consider the category
  $\rep(A_n,\bbF)$ of $\bbF$-linear representations of the quiver
  $A_n$.  It is not difficult to show that the lattice of subcategories
 satisfying any choice of the above closure properties does not depend on
  $\bbF$. We have chosen $\bbF=\bbF_2$ in order to emphasise the
  combinatorial nature of the counting problem.

  (3) Some choices are more natural than others. For instance, the
  finite colimit closed subcategories seem natural, given that
  for any ring $A$ the category $\mod A$ of finitely presented $A$-modules
  is the universal finitely cocomplete category over $A$; cf.\
  Proposition~\ref{pr:cocomplete}.
\end{rem}

\subsection*{Serre subcategories (Q,S,E)}
The Serre subcategories $\A\subseteq \A_n$ correspond to sets of simple
objects, by taking the simple objects of $\A_n$ that are contained in $\A$.
There are $n$ simple objects, and therefore the subsets form a
Boolean lattice of cardinality given by: \[2,4,8,16,32,64,\ldots \qquad 2^n\tag*{$\#(Q,S,E)$}\]

\subsection*{Thick subcategories (C,K,E)}
The thick subcategories are precisely the exact abelian and extension
closed subcategories.  Each thick subcategory $\A\subseteq \A_n$ is
generated by an exceptional sequence, and from this fact it follows
that the lattice of thick subcategories identifies with the lattice of
\emph{non-crossing partitions} $\NC(n+1)$; see
\cite{IT2009}. Explicitly, consider an $(n+1)$-gon with vertices
labeled $0,1,\ldots,n$ in clockwise order. The indecomposable objects $[a,b]$
that are simple in $\A$ correspond to chords $(a-1,b)$ describing a
non-crossing partition of order $n+1$.  Thus the number of thick
subcategories is given by the $(n+1)$-th Catalan number:
\[2, 5, 14, 42, 132, 429,\ldots \qquad C_{n+1}\tag*{$\#(C,K,E)$}\]

\subsection*{Quotient closed subcategories (Q)}

The quotient closed subcategories of $\A_n$ correspond to the elements
of the symmetric group $\mathfrak S_{n+1}$ with the sorting order
\cite{Ar2009,ORT2015}. In fact, the lattice equals the product of $n$
chains of length $2,\ldots ,n+1$; see also \cite{CFY2025} where such
subcategories are called pretorsion classes. Explicitly, we
identify
\[\mathfrak S_{n+1}=\{(a_1,\ldots,a_n)\in\bbZ^n\mid  0\le a_i\le
  i\text{ for }1\le i\le n\}\]
with partial order
\[ (a_i)\le (b_i)\quad :\iff\quad a_i\le b_i \text{ for }1\le i\le
  n.\] A quotient closed subcategory $\A\subseteq\A_n$ is determined
by the sequence of integers $(a_1,\ldots,a_n)$ where $a_i$ equals the
maximal $r\ge 0$ such that $P_i/\rad^r P_i$ is in $\A$. Here, $P_i$
denotes the indecomposable projective object of length $i$ and
therefore $0\le a_i\le i$ for all $i$. Conversely, for any such
sequence $(a_i)$ the direct sums of quotients of
$\bigoplus_{i=1}^n P_i/\rad^{a_i} P_i$ form a quotient closed
subcategory of $\A_n$.  Thus their number is given by:
\[2, 6, 24, 120, 720, 5040,\ldots \qquad (n+1)! \tag*{$\#(Q)$}\]

\subsection*{Quotient and subobject closed subcategories (Q,S)}

The quotient and subobject closed subcategories correspond to
sequences of integers $(a_1,\ldots,a_n)$ satisfying for $1\le i\le n$
that $0\le a_i\le i$ and in addition
\[a_{i+1}\le a_i +1.\] The second condition reflects the
property of a quotient closed subcategory to be closed under
subobjects. The number of such sequences is given by the $(n+1)$-th
Catalan number \cite[Exercise~80]{St2015}:
\[2, 5, 14, 42, 132, 429,\ldots \qquad C_{n+1}\tag*{$\#(Q,S)$}\]

Quotient and subobject closed subcategories that contain all simple
objects correspond to the sequences $(a_i)$ such that $a_i\ge 1$ for all
$i$, and their  number is given by the $n$-th
Catalan number.

\subsection*{Torsion classes (Q,E)}

The torsion classes given by a torsion pair are precisely the quotient
and extension closed subcategories. The lattice of such subcategories
identifies with the Tamari lattice of order $n+1$; see
\cite{Th2012}. Explicitly, torsion classes
correspond to sequences of integers $(a_1,\ldots,a_n)$ satisfying for
$1\le i\le j\le n$ that $0\le a_i\le i$
and in addition
\[a_j\ge j-i\quad\implies \quad a_j-a_i\ge j-i.\] The second condition
reflects the property of a quotient closed subcategory to be closed
under extensions. Thus their number is given by the $(n+1)$-th Catalan
number \cite[Problem~A41]{St2015}:
\[2, 5, 14, 42, 132, 429,\ldots \qquad C_{n+1}\tag*{$\#(Q,E)$}\]

Of particular interest are the tilting torsion classes. These are
torsion classes that are generated by a tilting object, and this sets
up a bijection with the isomorphism classes of basic tilting objects, which are counted by
the $n$-th Catalan number \cite{Ga1981}. The lattice is isomorphic to the Tamari
lattice of order $n$; see \cite{BK2004}.

\begin{rem}
  Consider integer sequences $(a_1,\ldots,a_n)$ satisfying
  $0\le a_i\le i$ for all $i$. We may substitute $b_i=i-a_i$ and then
  the torsion classes correspond to sequences $(b_i)$ such that for
  $1\le i\le j\le n$
  \[ b_j\le i\quad\implies \quad b_j\le b_i.\] Using the same
  substitution the quotient and subobject closed subcategories
  correspond to sequences $(c_i)$ such that $c_{i+1}\ge c_i$ for all
  $i$. There is an obvious map which sends a sequence $(b_i)$
  corresponding to a torsion pair to the weakly increasing sequence
  $(c_i)$ such that their multisets of values coincide, so
  $\{b_i\mid 1\le i\le n\}=\{c_i\mid 1\le i\le n\}$. This is actually
  a bijection, but it would be interesting to have an explicit
  construction on the level of subcategories.
\end{rem}  

\subsection*{Extension closed subcategories (E)}

The extension closed subcategories are precisely the fully exact and
idempotent complete subcategories. For the number of these
subcategories no formula is known \cite{OEIS}. The first terms of
the sequence are:
\[2, 7, 34, 199, 1308, 9300,\ldots \tag*{$\#(E)$}\]

In order to indicate the nature of the counting problem, let us
provide a criterion for a subcategory $\A\subseteq\A_n$ to be
extension closed. An elementary calculation shows that it suffices to
check the extensions of indecomposable objects. Fix a pair of
intervals $x=[a,b]$ and $x'=[a',b']$.

\begin{lem}\label{le:ext}
There is a non-zero element $\eta\in\Ext_{\A_n}^1(x',x)$ if and
only if $a\le a'$, $b\le b'$, and $a'\le b+1$. In this case $\eta$ is given by
an exact sequence
\[0\lto x\lto y\oplus y'\lto x'\lto 0\]
where $y=[a,b']$ and  $y'=[a',b]$, with convention $[b+1,b]=0$.\qed
\end{lem}
Thus it suffices to check that $ y, y'\in\A$ provided that
$x,x'\in\A$. 

\subsection*{Exact abelian subcategories (C,K)}

The exact abelian subcategories are precisely the subcategories that
are closed under finite limits and colimits. Their number is given by
the sequence of large Schröder numbers:
\[ 2, 6, 22, 90, 394, 1806, \ldots \qquad S_{n}\tag*{$\#(C,K)$}\] This
observation seems to be new and the proof uses a reduction to the
cases (C,K,E) and (Q,S). For details we refer to
Propositions~\ref{pr:CK} and \ref{pr:Schroeder}.

Counting the exact  abelian subcategories amounts to counting
subcategories that are closed under taking kernels and cokernels. We
treat both cases separately, and because of the duality
\eqref{eq:dual} it suffices to
consider subcategories that are closed under cokernels.

\subsection*{Finite colimit closed subcategories (C)}

These are the additive subcategories closed under cokernels, because
any finite colimit is given by the cokernel of a morphism between
finite coproducts. For the number of these subcategories no formula is
known \cite{OEIS}. The first  terms of the sequence are:
\[ 2, 7, 37, 265, 2396, 26118,  \ldots \tag*{$\#(C)$}\]

The counting problem can be translated into a combinatorial
problem. This is based on the following criterion for a subcategory
$\A\subseteq\A_n$ to be closed under taking cokernels.

\begin{lem}\label{le:cok}
  For $\A \subseteq \A_n$ to be closed under taking cokernels, it
  suffices to be closed under taking cokernels of morphisms
  $\phi\colon x\to y$, where $x$ is indecomposable and $y$ is either
  indecomposable or a direct sum of two indecomposable objects.  Let
  $x=[a,b]$, $y=[c,d]$ or $y=[c,d]\oplus [c',d']$. Assuming that each
  summand of $\phi$ is non-zero, the cokernel equals $[b+1,d]$ in the
  first case, and
  \[[b+1,\min(d,d')] \oplus [\max(c,c'),\max(d,d')]\]
  in the second case, with convention $[b+1,b]=0$.
\end{lem}
\begin{proof}
  First, note that we can indeed assume that $x$ is indecomposable,
  say $x=[a,b]$. We can also assume that for $y=\bigoplus_{i=1}^r y_i$
  with $y_i=[c_i,d_i]$ the morphisms $x\to y_i$ are non-zero. Next, we
  look at the case where $r=2$ and $d_1\le d_2, c_1\le c_2$. It is
  convenient to work in $\rep(A_n,\bbF_2)$, using the equivalence
  $\A_n\iso\rep(A_n,\bbF_2)$ that identifies $x$ with $M_x$. Let us
  decompose $M_y$ as follows.  We choose elements $e_i(j)$ such that
  $M_{y_i}(j)=\bbF_2 e_i(j)$ for $i=1,2$ and $j\in
  \{1,\ldots,n\}$. Set
   \[v_1(j):=\begin{cases} 0& \text{if }1\le j\le c_1-1\text{ or } d_1+1\le j \le n,\\
    e_1(j)& \text{if }c_1\le j\le c_2-1,\\
    e_1(j)+e_2(j) &\text{if }c_2\le j \le d_1,\\
   \end{cases}\]
   and
   \[v_2(j):=\begin{cases} 0& \text{if }1\le j\le c_2-1\text{ or } d_2+1\le j \le n,\\
    e_2(j)& \text{if }c_2\le j\le d_2.\\
   \end{cases}\]
Then
\[M_y=(\bbF_2 v_1(n)\to \cdots \to \bbF_2 v_1(1))\oplus (\bbF_2
  v_2(n)\to\cdots\to \bbF_2v_2(1))\] and the image of $\phi$ lies in
the first of the above summands. Thus the cokernel is the direct sum
of $M_{y_2}$ and the cokernel of $M_x\to M_{y_1}$. Now consider the
general case $r\ge 2$. We may arrange the summands such that
$d_1\le \cdots \le d_r$ and by the above argument for $r=2$ we can
also assume that $c_1\ge \cdots\ge c_r$. In this case the
$v_i(j)=\sum_{k=i}^re_k(j)$ yield a basis of the cokernel and therefore
   \[\Coker\phi= [b+1,d_1]\oplus\left(\bigoplus_{i=2}^r[c_{i-1},d_i]\right).\]
   It is easily checked that each indecomposable direct summand  also
   arises from the cokernel of
   a morphism of the form $x\to y_1\oplus y_2$, which proves the lemma.
 \end{proof}

\subsection*{Cokernel and extension closed subcategories (C,E)}

These subcategories have been studied by Enomoto
\cite{En2021,En2022}. In particular he shows that their number is
given by the sequence of large Schröder numbers:
\[ 2, 6, 22, 90, 394, 1806, \ldots\qquad S_n \tag*{$\#(C,E)$}\] When
counting these subcategories an interesting parallel to the case (C,K)
appears, because one can use an analogous reduction to the cases
(C,K,E) and (Q,E). For details we refer to Propositions~\ref{pr:CE}
and \ref{pr:Schroeder}.

Cokernel and extension closed subcategories are in bijection to
isomorphism classes of basic rigid objects, by taking such a
subcategory to a minimal projective generator; see
\cite[Theorem~A]{En2022} or the proof of
Proposition~\ref{pr:CE}. Recall that an object $X$ is \emph{rigid} if
$\Ext^i(X,X)=0$ for all $i>0$. Rigid modules are also known as
\emph{partial tilting modules}, and their combinatorial properties
have been studied extensively \cite{MRZ2003,RS1991}.

\subsection*{Additive subcategories}

The additive subcategories correspond to sets of indecomposable
objects. The cardinality of $\I_n$ equals $\frac{1}{2} (n^2 +n)$, and therefore the subsets of
$\I_n$ form a Boolean lattice of cardinality given by:
\[2, 8, 64, 1024, 32768, 2097152,\ldots \qquad 2^{\frac{1}{2} (n^2
    +n)}\tag*{$\#(\varnothing)$}\]

\subsection*{Computations}

The new sequences $\#(E)$, $\#(C)$, and $\#(C,K)$ have been computed
as follows. Any of these subcategories $\A\subseteq\A_n$ is determined
by its indecomposable objects. Thus for any subset of the
${\frac{1}{2} (n^2 +n)}$ indecomposables of $\A_n$ one needs to check that the
relevant closure properties are satisfied, using Lemmas~\ref{le:ext}
and \ref{le:cok}. This has been done by the authors for $n\le 6$ on a
personal computer via a programme written in Python.

\begin{rem}
  After the completion of this work, Frédéric Chapoton was able to
  confirm the computations of $\#(E)$, adding the values for
  $n=7,8$. Also, Volodymyr Mazorchuk confirmed the computations of
  our new sequences, after pointing out an error in a previous
  calculation. For further substantial progress on $\#(E)$ see
  \cite{Ma2026}.
\end{rem}

\section{A categorification of large Schröder numbers}

The correspondence between thick subcategories of the abelian category
$\A_n$ and the elements of the lattice $\NC(n+1)$ of non-crossing
partitions can be thought of as a \emph{categorification} of
$\NC(n+1)$; see \cite{HK2016} for details and a broader
perspective involving instead of the symmetric group $\mathfrak S_n$ all crystallographic
Coxeter groups. In the following we pass from the integer sequence
\[\#(C,K,E)(n)=\card \NC(n+1)=C_{n+1}\] to  $\#(C,K)$ which involves the
sequence of large Schröder numbers. We offer a proof that was found
with the help of AI.

For a permutation $\pi$ of cycle type $(\lambda_1,\ldots,\lambda_r)$
we set
\[C_\pi:=C_{\lambda_1-1}C_{\lambda_2-1}\cdots C_{\lambda_r-1}.\]

\begin{prop}\label{pr:CK}
 For an integer $n\ge 1$ we have
 \[\#(C,K)(n)= \sum_{\pi\in\NC(n+1)}C_\pi.\]
\end{prop}
\begin{proof}
  Let $\A\subseteq\A_n$ be a full additive subcategory closed under
  kernels and cokernels. We denote by $\bar\A$ the smallest thick
  subcategory of $\A_n$ that contains $\A$. From \cite{IT2009} it
  follows that $\bar\A$ corresponds to a non-crossing partition
  \[\pi\in\NC(n+1)\subseteq\mathfrak S_{n+1}.\] Writing
  $\pi=\pi_1\cdots\pi_r$ as a product of disjoint cycles, each $\pi_i$
  corresponds to a connected component of $\bar\A$ equivalent to
  $\A_{|\pi_i|-1}$. Thus if $(\lambda_1,\ldots,\lambda_r)$ denotes the
  cycle type of $\pi$, then there is an equivalence
  \[\bar\A\simeq\prod_{i=1}^r\A_{\lambda_i-1}.\]
  
  Next observe that the simple objects of $\A$ and $\bar\A$
  coincide. The kernel and cokernel closed subcategories of $\bar\A$
  that contain all simple objects are actually closed under subobjects
  and quotients. This is easily checked by an induction on the
  composition length.  For any $r\ge 1$ the computation of $\#(Q,S)$
  shows that subcategories of $\A_r$ closed under subobjects and
  quotients are given by integer sequences
  $(a_1,\ldots,a_r)$ satisfying $0\le a_i\le i$ and $a_{i+1}\le a_i+1$
  for all
  $i$. The additional property that a subcategory contains all simple
  objects means that $a_i\ge 1$ for all $i$. Thus the number of these
  integer sequences equals $C_r$. We conclude that for $\bar\A$
  corresponding to $\pi$ the subcategories closed under kernels and
  cokernels and containing all simples are counted by $C_\pi$. For an
  arbitrary choice of $\A$ any thick subcategory of $\A_n$ may occur
  as $\bar\A$, and therefore $\#(C,K)$ is given by
  $\sum_{\pi\in\NC(n+1)}C_\pi$.
\end{proof}

The following is an analogue of Proposition~\ref{pr:CK} and
takes care of subcategories of type (C,E). The result could be deduced
from \cite{En2021,En2022}, but we include a direct proof in order to
point out an intriguing parallel to type (C,K).

\begin{prop}\label{pr:CE}
 For an integer $n\ge 1$ we have
 \[\#(C,E)(n)= \sum_{\pi\in\NC(n+1)}C_\pi.\]
\end{prop}
\begin{proof}
  Let $\A\subseteq\A_n$ be a full additive subcategory closed under
  cokernels and extensions. We denote by $\bar\A$ the smallest thick
  subcategory of $\A_n$ that contains $\A$. As in the proof of
  Proposition~\ref{pr:CK} the subcategory $\bar\A$ corresponds to a
  non-crossing partition $\pi\in\NC(n+1)$, and if
  $(\lambda_1,\ldots,\lambda_r)$ denotes the cycle type of $\pi$, then
  there is an equivalence
  \[\bar\A\simeq\prod_{i=1}^r\A_{\lambda_i-1}.\]

  The category $\A$, viewed as an exact category,  has enough projective
  objects since it has only finitely many indecomposable objects. A
  projective generator  yields a tilting object in $\bar\A$ since it is
  rigid and generates, so $\bar\A$ equals the smallest thick
  subcategory containing this object. Then a
  standard argument shows that $\A$ is a torsion class in $\bar\A$,
  which is generated by a tilting object. For any $r\ge 1$ the
  computation of $\#(Q,E)$ shows that the tilting torsion classes of
  $\A_r$ are counted by $C_r$. We conclude that for $\bar\A$
  corresponding to $\pi$ the subcategories closed under cokernels and
  extensions and containing a tilting object are counted by
  $C_\pi$. For an arbitrary choice of $\A$ any thick subcategory of
  $\A_n$ may occur as $\bar\A$, and therefore $\#(C,E)$ is given by
  $\sum_{\pi\in\NC(n+1)}C_\pi$.
\end{proof}

For $n\ge 0$ let $S_n$ denote the $n$-th \emph{large Schröder number},
which is determined by the recursion
\[S_0=1\qquad\text{and}\qquad S_n=S_{n-1}+\sum_{k=0}^{n-1}S_k
  S_{n-1-k}.\]
Thus the generating function $S(x)=\sum_{n\ge 0}S_nx^n$ is the unique
power series satisfying
\[S(0)=1\qquad\text{and}\qquad S(x)=1+xS(x)+xS(x)^2.\]

The following results is known from the theory of free probability
\cite[Theorem~8.4]{Dy2007}. For the convenience of the reader we
include a direct proof; it was found using the help of AI.

\begin{prop}\label{pr:Schroeder}
For $n\ge 0$ we have
 \[S_n=\sum_{\pi\in\NC(n+1)}C_\pi.\] 
\end{prop}
\begin{proof}
  For $n\ge 0$ set \[f_n:=\sum_{\pi\in\NC(n)}C_\pi.\] A non-crossing
  partiton $\pi\in\NC(n)$ can be decomposed into smaller partitions by
  fixing the cycle containing $1$.  If this cycle has length $k$ then
  it yields $k$ intervals of length $n_i$ such that $\pi$ restricts to
  non-crossing partitions $\pi_i\in\NC(n_i)$. This gives a bijection
\[  \NC(n)\ \cong\ \bigsqcup_{k=1}^n\ \bigsqcup_{n_1+\cdots+n_k=n-k}
  \NC(n_1)\times\cdots\times \NC(n_k).\] Under this  bijection one gets
$C_\pi=C_{k-1}\prod_i C_{\pi_i}$ and therefore 
\begin{equation}\label{eq:f_n}
  f_n=\sum_{k=1}^n C_{k-1}\sum_{\substack{n_1+\cdots+n_k=n-k\\
      n_i\ge0}} f_{n_1}\cdots f_{n_k}.
\end{equation}
Consider the generating functions
\[F(x):=\sum_{n\ge 0}f_nx^n\qquad\text{and}\qquad
  C(x):=\sum_{n\ge 0}C_nx^n.\]
Multiply \eqref{eq:f_n} by $x^n$, sum over $n\ge1$, and add
$f_0=1$. This gives
\begin{align*}
  F(x)&=1+\sum_{k\ge1}C_{k-1}x^k\sum_{n\ge k}x^{n-k}\sum_{n_1+\cdots+n_k=n-k}f_{n_1}\cdots f_{n_k}\\
      &=1+\sum_{k\ge1}C_{k-1}\big(xF(x)\big)^k,
\end{align*}
since the inner double sum is exactly the coefficient extraction giving $F(x)^k$. Writing $j=k-1$,
\[\sum_{k\ge1}C_{k-1}\big(xF(x)\big)^k=xF(x)\sum_{j\ge0}C_j\big(xF(x)\big)^j=xF(x) C\big(xF(x)\big).\]
Thus $F$ satisfies
\[F(x)=1+xF(x)C\big(xF(x)\big)\] 
and this yields
\begin{equation}\label{eq:F}
  F^*(x):=\sum_{n\ge
    0}f_{n+1}x^n=\frac{F(x)-1}{x}=F(x)C\big(xF(x)\big).
\end{equation}  
The Catalan generating function satisfies the equation \[C(u) = 1 + uC(u)^2.\]
Substituting $u= xF(x)$  and using \eqref{eq:F} twice,
\begin{align*}
  F^*(x) &= F(x)\Big(1 + xF(x) C\big(xF(x)\big)^2\Big)\\
       &= F(x) + x \Big(F(x) C\big(xF(x)\big)\Big)^2\\
       &= F(x) + xF^*(x)^2\\
       &= 1 + xF^*(x) + xF^*(x)^2,
\end{align*}
where the last step uses $F(x) = 1 + xF^*(x)$. Thus the generating
function $F^*$ with coefficients $f_{n+1}$ satisfies
\[F^*(0)=1\qquad\text{and}\qquad F^*(x) = 1 + xF^*(x) + xF^*(x)^2,\]
which implies $f_{n+1}=S_n$ for all $n\ge 0$.
\end{proof}

\begin{cor}\label{pr:CK/CE}
 For an integer $n\ge 1$ we have
 \[\#(C,K)(n)= S_n=\#(C,E)(n).\]
\end{cor}
\begin{proof}
Combine Propositions~\ref{pr:CK} and \ref{pr:CE} with the identity for
the large Schröder numbers from  Proposition~\ref{pr:Schroeder}.
\end{proof}

\section{A universal category of poset representations}

Let $P=(P,\le)$ be a poset. We propose a canonical construction which
embeds $P$ into a finitely cocomplete additive category; it is abelian
when $P$ is finite. Recall that a category is \emph{finitely
  cocomplete} if it admits finite colimits. It is convenient to view
$P$ as a category: the objects are the elements of $P$ and there is a
unique morphism $x\to y$ if and only if $x\le y$.

\subsection*{A universal category of representations}
We begin by adjoining to $P$ a distinguished \emph{zero object}, so an
object which is both initial and terminal. The new category $P_*$ is
given by
\[\Ob P_*:=\Ob P\sqcup\{*\}\]
and
\begin{equation*}\label{eq:P*}
  \Hom_{P_*}(x,y):=\begin{cases} \bbF_2&\text{if }x\le y \text{ in } P,\\
    \{0\}& \text{otherwise}.
  \end{cases}
\end{equation*}
Taking for $x\le y$ the unique morphism $x\to y$ in $P$ to the
non-zero morphism in $P_*$ yields a canonical functor $P\to P_*$.

\begin{lem}
  The category $P_*$ is $\bbF_2$-linear and $*$ is the unique and
  therefore distinguished zero
  object. Any functor $P\to\C$ to a category with a distinguished zero
  object 
  extends uniquely to a functor $P_*\to\C$ that preserves distinguished zero
  objects.
\end{lem}
\begin{proof}
  The first assertion is clear.  We may view $P$ as a subcategory of
  $P_*$ and any morphism not in $P$ is a zero morphism.  Thus a
  functor $P\to\C$ extends uniquely to $P_*$ by sending zero morphisms
  to zero morphisms.
\end{proof}

The lemma implies that the category $P_*$ is \emph{preadditive}, so
enriched over the category $\Ab$ of abelian groups. For such
categories we consider functors that are additive. Recall that a
functor $F$ is \emph{additive} if for any pair of objects $x,y$ the
induced map
\[\Hom(x,y)\lto\Hom(F(x),F(y))\] is a group homomorphism.

Let $\Mod P_*$ denote the category of additive functors
$P_*^\op\to\Ab$ into the category of abelian groups. For a pair of
objects $F,G$ in $\Mod P_*$ we write $\Hom_{P_*}(F,G)$ for the set of
morphisms $F\to G$ which are given by the natural transformations.
The Yoneda embedding
\[P_*\lto\Mod P_*,\qquad x\mapsto h_x:=\Hom_{P_*}(-,x)\] provides
the \emph{free cocompletion} of $P_*$, so the universal additive functor
from $P_*$ to an $\Ab$-enriched cocomplete category.  Corestricting to the full
subcategory of finitely presented functors yields the \emph{finite
  cocompletion}
\[P_*\lto \mod P_*.\] 
Thus $\mod P_*$ is obtained from $P_*$ by adding finite coproducts
and cokernels, keeping in mind that any finite colimit in an additive
category can be written as the cokernel of a morphism between finite
coproducts.

The above constructions do not depend on the choice of any ring of
coefficients, even though $\bbF_2$ appears. We write
\[\rep P:=\mod P_*\qquad\text{and}\qquad \Rep P:=\Mod P_*\] and one may
think of these as \emph{universal categories of poset
  representations}.

Assume that $P$ is finite. Then we consider the \emph{incidence
  algebra} 
\[\bbF_2P:=\bigoplus_{x,y\in P}\Hom_{P_*}(x,y)\] with coefficients in
$\bbF_2$ and multiplication induced by the composition of morphisms in
$P_*$.  We write $\mod \bbF_2 P$ for its category of finitely
presented modules.

\begin{prop}\label{pr:cocomplete}
  For a poset $P$ consider the canonical functor $P\to\rep P$. The
  category $\rep P$ is an $\bbF_2$-linear additive category which is
  finitely cocomplete and has finite dimensional morphism spaces. When
  $P$ is finite there is a canonical equivalence
  $\rep P\iso\mod \bbF_2 P$ and $\rep P$ is abelian.
\end{prop}
 
\begin{proof}
  The first assertion is immediate from the construction. Now assume
  that $P$ is finite.  For $F$ in $\rep P$ there is a canonical right
  action of $\bbF_2 P$ on $\bigoplus_{x\in P} F(x)$. This yields a
  functor $\rep P\to\mod \bbF_2 P$. Now fix a (multiplicative) basis
  of $\bbF_2 P$ consisting of elements
  $\alpha_{xy}\in\Hom_{P_*}(x,y)$. Taking an $\bbF_2 P$-module $M$ to the
  representation of $P$ given by \[x\mapsto M(x):=M\alpha_{xx}\] and for
  $x\le y$ the linear map
  \[M(\alpha_{xy})\colon M(y)\lto M(x)\] given by right multiplication with $\alpha_{xy}$ yields a
  quasi-inverse. It remains to note that $\rep P$ is abelian since the
  algebra $\bbF_2 P$ is finite dimensional.
\end{proof}

We continue with further remarks, in particular to clarify the
universality of the functor $P\to \rep P$.

\begin{rem}
(1) The assignment $P\mapsto \rep P$ is functorial. Thus any poset morphism
$P\to Q$ induces a right exact functor $\rep P\to\rep Q$. The functor
$P\to \rep P$ is a composite of two universal functors, but there
seems to be no obvious universal property of the composite.

(2) We are not claiming that our universal construction lands in
finitely cocomplete additive categories. This may be confusing, as
$\rep P$ is an additive category, but a functor between additive
categories that preserves finite colimits and zero objects need not be
additive. An example arises for any field $\bbF$ of characteristic
different from $2$ by extending the obvious map $\bbF_2\to\bbF$ to
finite dimensional vector spaces.

(3) Given posets $P$ and $Q$, there is an isomorphism $P\iso Q$ if and
only if there is an equivalence $\rep P\iso\rep Q$. One direction is
clear, and to see that $\rep P$ determines $P$ one notes that
$P\to\rep P$ yields a bijection between the elements of $P$ and the
isoclasses of indecomposable projective objects in $\rep P$. Moreover,
$x\le y$ in $P$ if and only if $\Hom_{P_*}(h_x,h_y)\neq 0$.
\end{rem}

\begin{exm}
  For $n\ge 1$ consider $P_n=\{1,\ldots,n\}$ with the natural
  ordering. The assignment $i\mapsto [1,i]$ for $1\le i\le n$ yields a
  fully faithful functor $(P_n)_*\to\A_n$ which induces an equivalence
  $\rep P_n\iso \A_n$. A quasi-inverse maps an interval $[a,b]$ to the
  quotient $h_b/h_{a-1}$.
\end{exm}

\subsection*{Linear representations}

Let $\bbK$ be a commutative ring and write $\Mod\bbK$ for the category
of $\bbK$-modules. The forgetful functor $\Mod\bbK\to \Set$ admits a
left adjoint that takes a set $X$ to the free $\bbK$-module $\bbK X$
with basis $X$. For any category $\C$ let $\bbK\C$
denote its \emph{$\bbK$-linearisation} which is obtained by setting
\[\Ob\bbK\C:=\Ob\C,\]
and for any pair of objects $x,y$
\[\Hom_{\bbK\C}(x,y):=\bbK\Hom_\C(x,y).\] 
The \emph{$\bbK$-linear representations} of $\C$ are by definition
functors $\C\to\Mod\bbK$, and they identify with additive
functors $\bbK\C\to\Ab$.

For a poset $P$  we obtain categories of \emph{contravariant $\bbK$-linear representations}
by setting
\[\rep(P,\bbK):=\mod \bbK P\qquad\text{and}\qquad
  \Rep(P,\bbK):=\Mod\bbK P.\]
For $\bbK=\bbF_2$ the inclusion $\bbF_2 P\to P_*$ yields identifications
\[\rep P =\rep(P,\bbF_2)\qquad\text{and}\qquad
  \Rep P=\Rep(P,\bbF_2).\]

There is a tensor product
\[\Rep(P,\bbK)\times\Mod\bbK\lto \Rep(P,\bbK)\]
which is defined as follows. For $F\in\Rep(P,\bbK)$ and $M\in\Mod\bbK$
one defines $F\otimes_\bbK M$ by setting for $x\in P$
\[(F\otimes_\bbK M)(x):= F(x)\otimes_\bbK M.\]
For  $G\in\Rep(P,\bbK)$ there is a natural isomorphism
\begin{equation}\label{eq:tensor}
  \Hom_{\bbK P}(F\otimes_\bbK M,G)\cong\Hom_{\bbK}(M,\Hom_{\bbK
    P}(F,G)).
\end{equation}
This is clear when  $M=\bbK$ and the
general case follows by taking a free presentation of $M$.

\section{Coherent posets}

Let $P=(P,\le)$ be a poset and fix a field $\bbF$. In this section we
add some further aspects concerning the interplay between $P$ and its
category of $\bbF$-linear representations $\rep(P,\bbF)$. The
universal category $\rep P$ arises as a special case when
$\bbF=\bbF_2$. Specifically, we look at the lattice of subobjects for
objects in $\rep(P,\bbF)$.  This yields a description of the posets
such that $\rep(P,\bbF)$ is abelian. It is interesting to note that
these results do not depend on the field $\bbF$.

\subsection*{Subobjects}
An \emph{ideal} of $P$ is a subset $I\subseteq P$ that is downward
closed, so $x\le y$ in $P$ and $y\in I$ implies $x\in I$. The ideals
of $P$ are partially ordered by inclusion and form a complete
lattice. Note that the join is given by taking the union of subsets,
and the meet is given by intersections. Thus the lattice of ideals is distributive.

Given $x\in P$, we consider the
\emph{principal ideal}
\[\downarrow\! x:=\{y\in P\mid y\le x\}\]
and compare this with the representable functor
\[h_x:=\Hom_{\bbF P}(-,x)\] in $\Rep(P,\bbF)$.  For $y\in P$ we have
\[h_y\subseteq h_x\quad\iff\quad\Hom_{\bbF P}(h_y,h_x)\neq
  0\quad\iff\quad y\le x.\]

\begin{prop}\label{pr:subobj}
  The assignments
\[\downarrow\! x\supseteq I\mapsto \sum_{y\in
    I}h_y\subseteq h_x\qquad\text{and}\qquad
  h_x\supseteq F\mapsto \{y\in P\mid
F(y)\neq 0\}\subseteq {\downarrow\! x}
\]
yield mutually inverse and order preserving bijections between
\begin{enumerate}
\item the lattice of ideals of
  $\downarrow\! x$, and
\item the lattice of additive subfunctors of $h_x$.
\end{enumerate}
  In particular, both lattices are distributive. 
\end{prop}
\begin{proof}
  The maps are well defined. Then the assertion follows from the adjunctions
  \[F\subseteq \sum_{y\in
    I}h_y\quad\iff\quad \{y\in P\mid
  F(y)\neq 0\}\subseteq I\] and
  \[\sum_{y\in
    I}h_y\subseteq F\quad\iff\quad I\subseteq \{y\in P\mid
  F(y)\neq 0\}.\qedhere\]
\end{proof}

The proposition shows that chain conditions on subobjects in $\rep(P,\bbF)$
translate into chain conditions on the lattice of ideals of $P$.

\begin{cor}\label{co:subobj}
For a poset $P$ and a field $\bbF$ the following holds.
\begin{enumerate}
\item The category $\rep(P,\bbF)$ is abelian and all objects are of finite length if
and only if  for each $x\in P$ the set $\downarrow\! x$ is finite.
\item The category $\rep(P,\bbF)$ is abelian and all objects are noetherian if
and only if for each $x\in P$ all ascending chains of ideals of $\downarrow\! x$ stabilise.
\end{enumerate}  
\end{cor}
\begin{proof}
  For any ring there are well known characterisations in terms of
  chain conditions on right ideals such that the category
  of finitely presented right modules is an abelian
  category consisting either of finite length or of noetherian
  objects. The proof for module categories carries over to our
  setting, by applying
  Proposition~\ref{pr:subobj}. For (1) one uses in addition that a distributive
  lattice satisfies both chain conditions if and only if it is finite.
\end{proof}

\begin{rem}
  A subrepresentation of $F\in\rep(P,\bbF)$ need not be finitely
  presented. Thus the lattices of subobjects in $\rep(P,\bbF)$ and in
  $\Rep(P,\bbF)$ do not agree in general. In fact, they agree
  precisely in the noetherian case, so when for each $x\in P$ all ascending
  chains of ideals of $\downarrow\! x$ stabilise.
\end{rem}

\subsection*{Coherent posets}
For rings there is the following hierarchy:
\[\text{right artinian}\quad\implies\quad \text{right
    noetherian}\quad\implies\quad\text{right coherent}.\] The
condition `right coherent' means that the category of finitely
presented right modules is abelian. In our context we wish to characterise
the posets such that $\rep(P,\bbF)$ is abelian. To this end we need to
recall some finiteness conditions.

An element $x$ in a lattice is \emph{compact}
if $x\le \bigvee_{\alpha\in A} x_\alpha$ implies $x\le \bigvee_{\alpha\in B} x_\alpha$ for
some finite subset $B\subseteq A$.

We continue with conditions on objects in a Grothendieck category and
refer to \cite{St1975} for details. An object $X$ is \emph{finitely
  generated} if $X=\sum_{\alpha\in A} X_\alpha$ for any family of subobjects
$X_\alpha\subseteq X$ implies $X=\sum_{\alpha\in B} X_\alpha$ for some finite
subset $B\subseteq A$. A finitely generated object $X$ is
\emph{finitely presented} if for any epimorphism $X'\to X$ from a
finitely generated object $X'$ the kernel is finitely generated. A
finitely presented object is \emph{coherent} if every finitely
generated subobject is finitely presented. For example, any noetherian
object is coherent because all subobjects are finitely generated.

\begin{lem}\label{le:mod-coherent}
 For a Grothendieck category with a generating set of objects that are
  finitely generated and projective, the following  are equivalent.
  \begin{enumerate}
\item All finitely generated projective objects are coherent.
\item The finitely presented objects form an abelian category.
 \end{enumerate}
\end{lem}
\begin{proof}
  Adapt the proof for the category of modules over a ring.
\end{proof}

We propose the following definition and the subsequent result
justifies the terminology.

\begin{defn}
A poset $P$ is called \emph{coherent} if for every diagram $y\le x\ge
y'$ there are finitely many diagrams $y\ge z_\alpha\le y'$ such that for each
diagram $y\ge z\le y'$ there is an index $\alpha$ with  $z\le z_\alpha$. 
\end{defn}

\begin{prop}\label{pr:coherent}
  For a poset $P$ the following are equivalent.
\begin{enumerate}
\item The poset $P$ is coherent.
\item For each $x\in P$ the compact elements in the lattice of ideals
  of $P$ contained in $\downarrow\! x$ are closed under finite meets.
\item For each $x\in P$ the object $h_x\in\Rep(P,\bbF)$ is coherent.
\item The category $\rep(P,\bbF)$ is abelian.
\end{enumerate}
\end{prop}

\begin{proof}
  (1) $\Leftrightarrow$ (2): The ideals of a poset form a distributive
  lattice, and an ideal $I$ is compact if and only if
  $I=\bigcup_\alpha\downarrow\! x_\alpha$ for some finite set of elements $x_\alpha$.
  Thus to check condition (2) it suffices to show that $\downarrow\!
  y\cap {\downarrow\! y'}$ is compact for each pair $y,y'\le x$, and
  this is precisely condition (1).

  (2) $\Rightarrow$ (3): For a morphism
  $\phi\colon \bigoplus_{\alpha\in A} h_{y_\alpha}\to h_x$ given by a finite set
  of elements $y_\alpha\in {\downarrow\! x}$ we need to show that its
  kernel is finitely generated.  For $z\in P$ set
  \[V_z:=\Ker\left( \bigoplus_{\alpha\in A} h_{y_\alpha}(z)\to
      h_x(z)\right).\] The inclusion
  \[V_z\lto \bigoplus_{\alpha\in A} h_{y_\alpha}(z)=\Hom_{\bbF P}(h_z,
    \bigoplus_{\alpha\in A} h_{y_\alpha})\] corresponds via \eqref{eq:tensor}
  to a morphism
  \[\psi_z\colon h_z\otimes_\bbF V_z\lto \bigoplus_{\alpha\in A}
  h_{y_\alpha}.\] Using the adjunction \eqref{eq:tensor} it follows that
  $\phi\psi_z=0$ and any morphism
  $\psi\colon h_z\to \bigoplus_{\alpha\in A} h_{y_\alpha}$ satisfying
  $\phi\psi=0$ factors through $\psi_z$.

  Using the coherence of $P$ there is for each subset $B\subseteq A$ a
  finite subset $B'\subseteq P$ satisfying
\[\bigcap_{\alpha\in B}\downarrow\!
  y_\alpha=\bigcup_{z\in B'}\downarrow\! z.\]
This yields a morphism
\[\bigoplus_{B\subseteq A}\bigoplus_{z\in B'}
  h_z\otimes_\bbF V_z\xto{\,(\psi_z)\,} \bigoplus_{\alpha\in A}
  h_{y_\alpha},\] where $B\subseteq A$ runs through all subsets, and its
image equals $\Ker\phi$ by construction. In fact, for any
$\psi\colon h_z\to\Ker\phi$ choose $B=\{\alpha\in A\mid\psi_\alpha\neq 0\}$. Then
$z\le z'$ for some $z'\in B'$ and $\psi$ factors through $\psi_{z'}$.

(3) $\Rightarrow$ (1): A diagram $y\le x\ge
y'$ yields a morphism $h_y\oplus h_{y'}\to h_x$, and its kernel is
given by diagrams $y\ge z\le y'$. Suppose the kernel is
a  quotient of $\bigoplus_\alpha h_{z_\alpha}$  for some finite set of elements
  $z_\alpha\in\downarrow\! x$.  This yields finitely many diagrams $y\ge z_\alpha\le y'$ such that for each
diagram $y\ge z\le y'$ there is an index $\alpha$ with  $z\le z_\alpha$. 

(3) $\Leftrightarrow$ (4): This follows from
Lemma~\ref{le:mod-coherent}, once we observe that coherent objects are
closed under finite direct sums and summands.
\end{proof}

\begin{cor}
A poset is coherent if and only if for the lattice of ideals the
meet of two compact elements is again compact.\qed
\end{cor}

\subsection*{Acknowledgements}

The first named author is grateful to Ren\'e Marczinzik and Baptiste
Rognerud for valuable comments on this work. He is also grateful to
Max Planck Institute for Mathematics in Bonn for its hospitality and
financial support. Both authors wish to thank Volodymyr Mazorchuk for
his interest in this work and for pointing out an error in a previous
calculation. We thank an anonymous referee for several helpful
comments.  This work was supported by the Deutsche
Forschungsgemeinschaft (SFB-TRR 358/1 2023 - 491392403).

\end{document}